\documentclass[12pt]{amsart}
\usepackage{times}

\usepackage{float}
\usepackage{geometry,graphicx,fancybox}
\usepackage{enumerate,amsmath,amssymb,amsthm}
\usepackage{comment}
\usepackage{epsfig}
\usepackage{pst-grad}
\usepackage{pst-plot}
\usepackage{mathtools}
\usepackage{enumitem}
\usepackage{bbold}
\usepackage{mathrsfs}
\usepackage[colorlinks,linkcolor=blue]{hyperref}
\usepackage{ dsfont }

\theoremstyle{plain}

\usepackage[kerning=true]{microtype} 
\usepackage{graphicx}
\usepackage{wrapfig}
\usepackage{multirow}

\newtheorem{theorem}{Theorem}
\newtheorem{corollary}[theorem]{Corollary}
\newtheorem{lemma}[theorem]{Lemma}

\theoremstyle{remark}

\theoremstyle{definition}
\newtheorem{remark}[theorem]{Remark}

\newcommand{\Z}{\mathbb{Z}}
\newcommand{\F}{\mathbb{F}}
\newcommand{\R}{\mathbb{R}}

\newcommand{\Mtil}{\widetilde{M}}

\newcommand{\Gtil}{\widetilde{G}}

\newcommand{\xtil}{\widetilde{x}}
\newcommand{\ytil}{\widetilde{y}}
\newcommand{\ztil}{\widetilde{z}}

\newcommand{\eps}{\varepsilon}

\DeclareMathOperator{\vol}{vol}

\title{ A curvature-free Log(2k-1) Theorem}

\author{Florent Balacheff and Louis Merlin}

\address{F. Balacheff, Universitat Aut\`onoma de Barcelona, Spain.}

\email{fbalacheff@mat.uab.cat}

\address{L. Merlin, Rheinisch-Westfälische Technische Hochschule (RWTH), Aachen, Germany}

\email{louis.merlin@hotmail.fr}

\thanks{The first author acknowledges support by the European Social Fund and the Agencia Estatal de Investigaci\'on through the Ram\'on y Cajal grant RYC-2016-19334  "Local and global systolic geometry and topology" and the FEDER/MICINN grant PGC2018-095998-B-I00 . The second author was partially supported by UL IRP grant NeoGeo and FNR grants INTER/ANR/15/11211745 and OPEN/16/11405402 and also acknowledges support from U.S. National Science Foundation grants DMS-1107452, 1107263, 1107367 ``RNMS: GEometric structures And Representation varieties'' (the GEAR Network).}


\begin{document}

\maketitle

\begin{abstract}
This paper presents a curvature-free version of the $\text{Log}(2k-1)$ Theorem of Anderson, Canary, Culler \& Shalen \cite{ACCS}.  It generalizes a result by Hou \cite{Hou01} and its proof is rather straightforward once we know the work by Lim \cite{Lim08} on volume entropy for graphs. As a byproduct we obtain a curvature-free version of the Collar Lemma in all dimensions.
\end{abstract}




\bigskip

\bigskip

The celebrated $\text{Log}(2k-1)$ Theorem by Anderson, Canary, Culler \& Shalen  states that if $\Gamma$ is a Kleinian group freely generated by elements $\gamma_1,\ldots,\gamma_k$ then for any $\xtil \in \mathds{H}^3$
$$
\sum_{i=1}^k \frac{1}{1+e^{d(\xtil,\gamma_i\cdot \xtil)}}\leqslant\frac{1}{2}.
$$
In particular there exists some $i \in \{1,\ldots,k\}$ such that $d(\xtil,\gamma_i\cdot \xtil)\geq \log(2k-1)$.
Strictly speaking, they initially proved their result in  \cite{ACCS} under additional tameness and hyperbolicity assumptions which can be removed due to later advances in Kleinian groups theory, see \cite[Theorem 7.7]{Can08}. This result can be viewed as a quantitative generalization of the Margulis lemma and has been extended in \cite[Theorem 1.1]{Hou01} to complete Riemannian manifolds of negative pinched curvature with a term involving the critical exponent of the group. In this article we drop off the curvature assumption in Hou's result by showing the following curvature-free inequality.

 \begin{theorem}\label{thm:main}
Let $\Mtil$ be a simply connected and complete Riemannian manifold. If $\Gamma$ is a discrete subgroup of isometries of $\Mtil$ freely generated by elements $\gamma_1,\ldots,\gamma_k$ then for any $\xtil \in \Mtil$
$$
\sum_{i=1}^k \frac{1}{1+e^{\delta(\Gamma) \cdot d(\xtil,\gamma_i\cdot \xtil)}}\leqslant\frac{1}{2}
$$
where $\delta(\Gamma)$ denotes the critical exponent of $\Gamma$.
\end{theorem}

Recall that the critical exponent of $\Gamma$ is defined as the unique real number $\delta(\Gamma)$ such that the Poincar\'e series
$$
\sum_{\gamma \in \Gamma}\exp(-sd(\xtil,\gamma \cdot \xtil))
$$
converges for $s>\delta(\Gamma)$ and diverges for  $s<\delta(\Gamma)$. It does not depend on the chosen point $\xtil$, and by standard arguments on Dirichlet series
$$
\delta(\Gamma)=\lim_{R \to \infty} \frac{\log \, \#\{ \gamma \in \Gamma \mid d(\xtil,\gamma\cdot \xtil)\leqslant R\}}{R}.
$$

Theorem \ref{thm:main}  applies in particular for $\Mtil=\mathds{H}^3$ (just like Hou's result) and gives a stronger inequality than the  $\text{Log}(2k-1)$ Theorem when the critical exponent of the Kleinian group is less than $1$ (such examples do exist by \cite{Laz14}). But for Kleinian groups we only have the upper bound $\delta(\Gamma)\leq 2$ (more generally, the critical exponent of a discrete subgroup of isometries of $\mathds{H}^n$ is at most $n-1$) which provides a weaker inequality  than the original $\text{Log}(2k-1)$ Theorem.  \\

In the special case where $\Mtil$ is the universal cover of a closed Riemannian manifold $M$, it is classical (see \cite[Lemma 2.2]{Sab06}) that the critical exponent of the fundamental group $\pi_1 M$ identified with the deck transformation group coincides with a well known Riemannian invariant called {\it volume entropy} (or sometimes {\it asymptotic volume}) of $M$ and defined as the exponential growth rate of volume of balls in its universal cover through the following formula:
$$
h_{\vol}(M):=\lim_{R \to \infty} \frac{\log \, \vol B(\xtil,R)}{R}.
$$
Here $B(\xtil,R)$ denotes the metric ball of radius $R$ around $\xtil$ in $\Mtil$ and $\vol$ the Riemannian volume.
Note that for an element $\gamma \in \pi_1 M$ the distance $d(\xtil,\gamma\cdot \xtil)$ coincides with the length $\ell(c)$ of a shortest geodesic loop $c$ in the class $\gamma$ and based at $x$ (the projection of $\xtil$ on $M$ by the covering map). 
 As the critical exponent of a subgroup of $\pi_1M$ is bounded from above by the volume entropy of $M$, our main theorem implies the following.

\begin{corollary}\label{cor:1}
Let $M$ be a closed Riemannian manifold and $x \in M$.
Assume that there exists a family $c_1,\ldots,c_k$ of homotopically independent loops based at $x$.

Then the following inequality holds true:
$$
\sum_{i=1}^k \frac{1}{1+e^{h_{vol}(M) \cdot \ell(c_i)}}\leqslant\frac{1}{2}.
$$
\end{corollary}
 
 Remember that $k$ loops are said to be  {\it homotopically independent} if their homotopy classes generate a free subgroup of rank $k$.\\

Theorem \ref{thm:main} is optimal: for any $k\geqslant 2$ there exists a sequence of Riemannian metrics $(g_n)$ on the closed manifold $X$ formed by taking the connected sum of $k$ copies of $S^1 \# S^2$, a point $x$ on $X$ and a family of homotopically independent loops $c_1,\ldots,c_k$ based at $x$ generating the fundamental group $\pi_1 X\simeq \F_k$ such that 
$$
\lim_{n \to \infty} \sum_{i=1}^k \frac{1}{1+e^{\ell_{g_n}(c_i) \cdot h_{vol}(X,g_n)}}= \frac{1}{2}.
$$
Here the volume entropy coincides with the critical exponent of the subgroup of isometries of the universal cover $\tilde{X}$ associated with the $c_i$'s homotopy classes.
See Remark \ref{rem:g} for more details.\\

As another consequence of Theorem \ref{thm:main} we get the following.

\begin{corollary}\label{cor:collarlemma}
Fix $h>0$.
Let $M$ be a complete Riemannian manifold with volume entropy $h$.
Suppose that $c_1$ and $c_2$ are two homotopically independent loops based at $x$.

Then
$$
\ell(c_2)\geqslant {1\over h} \log\left({4\over h\ell(c_1)}\right) +o(1)
$$
for $\ell(c_1)$ sufficiently close to $0$.
\end{corollary}

So if the volume entropy is kept fixed while the shortest length shrinks to zero, then the length of the largest loop blows up and we control the rate of explosion. 
It partially recovers (albeit with a worst multiplicative constant) and also generalizes to free curvature metrics the classical consequence of the Collar Lemma \cite[Corollary 4.1.2]{Bus92} that given a closed hyperbolic surface $S$ and two simple closed geodesics $c_1$ and $c_2$ intersecting each other, then the following sharp inequality is satisfied:
$$
\sinh\left(\frac{\ell(c_1)}{2}\right)\sinh \left(\frac{\ell(c_2)}{2}\right)> 1.
$$
Indeed this inequality admits the following expansion
$$
\ell(c_2)\geqslant 2 \log \left(\frac{4}{\ell(c_1)}\right)+o(1)
$$
for $\ell(c_1) \to 0$ while $h_{\vol}(S)=1$.

After showing Corollary \ref{cor:collarlemma}, we discovered that a curvature-free analog of the Collar Lemma has been independently proved in \cite[Lemma 7.12]{BCGS} where they obtained that 
$$
\ell(c_2)> {1\over h} \log\left({1\over h\ell(c_1)}\right)
$$
under the same assumptions.
Our corollary is slightly better than their result for small values of $\ell(c_1)$, but most importantly Theorem \ref{thm:main} relates it to a more general inequality. Compare also with \cite[Theorem 1.2]{Cer14}.
\\

For large families of homotopically independent loops we can also bound from below the length of the largest one in terms of the previous ones.

\begin{corollary}\label{cor:collarlemma2}
Fix $h>0$ and $k\geq 3$.
Let $M$ be a complete Riemannian manifold with volume entropy $h$.
Suppose that $c_1,\ldots,c_k$ are homotopically independent loops based at $x$ ordered by increasing length: $\ell(c_1)\leqslant \ldots \leqslant \ell(c_k)$.

Then
$$
\ell(c_k)\geqslant {-1\over h} \log\left( {h\ell(c_1)\over 4}-\sum_{i=2}^{k-1} e^{-h \ell(c_i)} \right)+o(1)
$$
for $\ell(c_1)$ sufficiently close to $0$.
\end{corollary}

To illustrate which information provides this inequality, observe that if the second length is $\eps$-close to the lower bound in Corollary \ref{cor:collarlemma}, then the third length blows up at a speed at least $- \log \eps /h$.\\

We now prove Theorem \ref{thm:main}.
Let $\Gamma$ be a discrete subgroup of isometries of a simply connected and complete Riemannian manifold $\Mtil$ freely generated by elements $\gamma_1,\ldots,\gamma_k$. Fix some $\xtil \in \Mtil$ and set $a_i:=d(\xtil,\gamma_i\cdot \xtil)$ for $i=1,\ldots,k$.\\

Consider the metric graph $\Gtil$ defined as follows. The vertices of $\Gtil$ are in one-to-one correspondence with points $\{\gamma \cdot \tilde{x} \mid \gamma \in \Gamma\}$, and two vertices labelled by $\ytil$ and $\ztil$ are connected through an edge of length $a_i$ if and only if  $\ztil =\gamma_i^{\pm 1} \cdot \ytil$. As the action of $\Gamma$ on $\Mtil$ is free, this graph is an infinite tree of valence $2k$ and is the universal cover of the metric graph denoted by $G_{a_1,\ldots,a_k}$ defined as the wedge product of $k$ circles of respective lengths $a_1,\ldots,a_k$.
It is then easy to check that
$$
\#\{\tilde{v} \in V(\tilde{G}) \mid d_{\tilde{G}}(\tilde{x},\tilde{v}) \leqslant R\} \leq \#\{\gamma \cdot \tilde{x} \mid d(\tilde{x},\gamma \cdot \tilde{x})\leqslant R\}.
$$
It implies that
$$
h_{vol}(G_{a_1,\ldots,a_k})\leqslant \delta(\Gamma),
$$
and the announced inequality
$$
\sum_{i=1}^k {1 \over 1+e^{\delta(\Gamma) \cdot a_i}}\leqslant{1\over 2}
$$
is then a straightforward consequence of the following result for graphs.

\begin{lemma}\label{lemma:graph}
The volume entropy $h:=h_{vol}(G_{a_1,\ldots,a_k})$ satisfies the following equality:
$$
\sum_{i=1}^k {1 \over 1+e^{ha_i}}={1\over 2}.
$$
\end{lemma}

\begin{proof}
According to \cite[Theorem 4]{Lim08} we know that $h$ is the only positive real number such that the following linear system of equations with unkowns $x_i$ 
$$
\left\{
\begin{array}{c}
x_1=x_1e^{-ha_1}+2x_2e^{-ha_2}+\ldots+2x_ke^{-ha_k}   \\
x_2=2x_1e^{-ha_1}+x_2e^{-ha_2}+\ldots+2x_ke^{-ha_k}\\
\vdots\\
x_k=2x_1e^{-ha_1}+2x_2e^{-ha_2}+\ldots+x_ke^{-ha_k}
\end{array}
\right.
$$
has a solution with $x_i>0$ for $i=1,\ldots,k$.
So take such a solution $(x_1,\ldots,x_k) \in (\R^\ast_+)^n$.
By summing all the equations we see that
$$
\sum_{i=1}^k x_i=\sum_{i=1}^k (2k-1)e^{-ha_i}x_i,
$$
and by substracting any two different lines $(L_i)$ and $(L_j)$ we get that
$$
(1+e^{-ha_i})x_i=(1+e^{-ha_j})x_j.
$$
So
$$
\sum_{i=1}^k (1- (2k-1)e^{-ha_i})x_i=0
$$ which implies that
$$ 
\sum_{i=1}^k \frac{1-(2k-1)e^{-ha_i}}{1+e^{-ha_i}}=0.
$$
We then easily derive the announced equality
$$
\sum_{i=1}^k {1 \over 1+e^{ha_i}}={1\over 2}.
$$
\end{proof}

\begin{remark}\label{rem:g}
Theorem \ref{thm:main} is optimal: for any $k\geqslant 2$, there exists a sequence of Riemannian metrics $(g_n)$ on the connected sum $X$ of $k$ copies of $S^1 \times S^2$, a point $x$ on $X$ and a family of homotopically independent loops $c_1,\ldots,c_k$ based at $x$ generating the fundamental group such that 
$$
\lim_{n \to \infty} \sum_{i=1}^k \frac{1}{1+e^{\ell_{g_n}(c_i) \cdot h_{vol}(X,g_n)}}= \frac{1}{2}.
$$
The construction of the sequence of metrics $(g_n)$ can be easily obtained by slightly modifying the simplicial Riemannian metric defined on the wedge product of $k$ copies of $S^1 \times S^2$ as follows. Fix $k$ positive real numbers $a_1,\ldots,a_k$ and for $i=1,\ldots,k$ consider on each copy $(S^1 \times S^2)_i$ the metric product $a_i^2 dt^2 \otimes ds$ where $dt^2$ denotes the standard Riemannian metric on $S^1$ of length $1$ and $ds$ the standard Riemannian metric on $S^2$ of area $4\pi$. Then the simplicial Riemannian metric $g$ induced on the wedge product $\vee_{i=1}^k (S^1 \times S^2)_i$ has the following property. If $x$ denotes the common point to all factors, any minimal (in its homotopical class) geodesic loop $\gamma$ based at $x$ decomposes as a unique concatenation $\alpha_1\star \ldots \star \alpha_N$ where each $\alpha_j$ is a minimal geodesic loop in some factor $(S^1 \times S^2)_{i_j}$ and whose class is the $p_j$-iterated for some $p_j \in \Z\setminus\{0\}$ of a generator of the fundamental group of this factor. It is thus straighforward to see that $\ell(\gamma)=\sum_{j=1}^N |p_j|\cdot a_{i_j}$ from which we deduce that the volume entropy of $(\vee_{i=1}^k (S^1 \times S^2)_i,g)$ is equal to the volume entropy of $G_{a_1,\ldots,a_n}$. Now observe that we can choose for each $i=1,\ldots,k$ as loop $c_i$ any minimal geodesic loop based at $x$ and contained in $(S^1 \times S^2)_i$ that corresponds to a generator of the fundamental group of this factor. In particular $\ell_{g}(c_i)=a_i$ for $i=1,\ldots,k$ so
$$
\sum_{i=1}^k \frac{1}{1+e^{\ell_{g}(c_i) \cdot h_{vol}(\vee_{i=1}^k (S^1 \times S^2)_i,g))}}= \frac{1}{2}.
$$
Finally we construct the sequence of metrics $(g_n)$ to smooth out the base point $x$. Indeed we see the metric $g$ as a singular metric on the connected sum $\#_{i=1}^k (S^1 \times S^2)_i$ and we approximate $g$ in the $C^0$-topology by smooth Riemannian metrics. The conclusion follows as both length and volume entropy are continuous maps for this topology.
\end{remark}

\medskip

\noindent    \textbf{Acknowledgments.} We would like to thank Yi Huang and the two referees for valuable comments.


\end{document}